\documentclass[12pt,reqno]{amsart}

\usepackage{color}
\usepackage{mathrsfs}
\usepackage{enumitem}
\usepackage{graphicx}
\usepackage[normalem]{ulem}
\usepackage{amsmath, amssymb, amsfonts, dsfont, amsthm}
\usepackage[utf8]{inputenc}
\usepackage{xcolor}
\DeclareGraphicsExtensions{.pdf,.png,.jpg}
\newtheorem{thm}{Theorem}[section]

\newtheorem{lemma}[thm]{Lemma}
\newtheorem{prop}[thm]{Proposition}

\theoremstyle{remark}
\newtheorem{remark}[thm]{Remark}

\usepackage{anyfontsize}
\usepackage{tikz}
\usetikzlibrary{calc,intersections,arrows.meta,patterns}

\def\XXint#1#2#3{{\setbox0=\hbox{$#1{#2#3}{\int}$}
\vcenter{\hbox{$#2#3$}}\kern-.5\wd0}}

\newcommand\cbrk{\text{$]$\kern-.15em$]$}}
\newcommand\opar{\text{\,\raise.2ex\hbox{${\scriptstyle
|}$}\kern-.34em$($}}
\newcommand\cpar{\text{$)$\kern-.34em\raise.2ex\hbox{${\scriptstyle |}$}}\,}

\def\<{\langle}
\def\>{\rangle}

\newcommand\bR{\mathbb{R}}

\newcommand\cD{\mathcal{D}}

\newcommand\cL{\mathcal{L}}

\newcommand\cM{\mathcal{M}}

\newcommand{\mysection}[1]{\section{#1}
\setcounter{equation}{0}}

\DeclareFontFamily{U}{matha}{\hyphenchar\font45}
\DeclareFontShape{U}{matha}{m}{n}{
      <5> <6> <7> <8> <9> <10> gen * matha
      <10.95> matha10 <12> <14.4> <17.28> <20.74> <24.88> matha12
      }{}
\DeclareSymbolFont{matha}{U}{matha}{m}{n}
\DeclareFontSubstitution{U}{matha}{m}{n}
\DeclareFontFamily{U}{mathx}{\hyphenchar\font45}
\DeclareFontShape{U}{mathx}{m}{n}{
      <5> <6> <7> <8> <9> <10>
      <10.95> <12> <14.4> <17.28> <20.74> <24.88>
      mathx10
      }{}
\DeclareSymbolFont{mathx}{U}{mathx}{m}{n}
\DeclareFontSubstitution{U}{mathx}{m}{n}
\DeclareMathDelimiter{\vvvert}{0}{matha}{"7E}{mathx}{"17}

\definecolor{felix}{rgb}{0.2,0.2,1.0} 
\definecolor{petru}{rgb}{0.7,0.1,0.1} 

\begin{document}

\title[Green's funtion with wedge boundary]{A refined Green's function estimate of the time measurable parabolic operators with conic domains}

\author{Kyeong-Hun Kim}
\thanks{The first and  third authors were  supported by the National Research Foundation of Korea(NRF) grant funded by the Korea government(MSIT) (No. NRF-2019R1A5A1028324)}
\thanks{The second author was  supported by the National Research Foundation of Korea(NRF) grant funded by the Korea government(MSIT) (No. NRF-2019R1F1A1058988)}
\address{Kyeong-Hun Kim, Department of Mathematics, Korea University, 1 Anam-Dong, Sungbuk-gu, Seoul, 136--701, Republic of Korea}
\email{kyeonghun@korea.ac.kr}

\author{Kijung Lee}
\address{Kijung Lee, Department of Mathematics, Ajou University, Suwon, 443--749, Republic of Korea} 
\email{kijung@ajou.ac.kr}

\author{Jinsol Seo}
\address{Jinsol Seo, Department of Mathematics, Korea University, 1 Anam-Dong, Sungbuk-gu, Seoul, 136--701, Republic of Korea}
\email{seo9401@korea.ac.kr}

\subjclass[2010]{34B27, 35K08, 35K20, 60H15}

\keywords{Green's function estimate, conic domains, wedge domains, stochastic parabolic equation}

\begin{abstract}
We present a refined Green's function estimate of the time measurable parabolic operators on conic domains  that involves mixed weights consisting of  appropriate powers of the distance to the vertex and of the distance to the boundary.
\end{abstract}

\maketitle

\mysection{Introduction}
\label{sec:Introduction}
In recent years  we have been interested in the stochastic heat diffusion occurring  in wedge shaped subdomains of $\mathbb{R}^2$, which are
probably simplest non-smooth Lipschitz domains. In the literature   there exist almost fully developed regularity results for the stochastic heat diffusion on $C^1$ domains, but when it comes to Lipschitz domains  the results are quite  unsatisfactory and very little is known. To fill in the gap between the theory for $C^1$  domains and the theory for Lipschitz domains,  the wedge domains  are what we decided to pay attention first. 

Along the way, we set the theme that the angle around the vertex affects regularity of the temperature when the boundary temperature is controlled. We believe that our previous work \cite{CKLL 2018} captured such relation in a certain way. Based on this work, in \cite{CKL 2019+}  we proceeded to construct a theory on the stochastic diffusion in polygonal domains.
The main tool of our results was an estimate on Green's function for the  heat operator with the wedge domains obtained in \cite{Koz1991}. Looking back,  what we feel sorry about is that the estimate only involves the weight of powers of the distance to the vertex.  ``only'' means that it could be better or much better if the estimate also involves weight of the distance to the boundary. Having weight depending only on  the distance to the vertex in the estimate did not yield satisfactory boundary regularity of the solution and caused quite a bit of  trouble when we constructed a global regularity theory for polygonal domains.  

Aiming more natural and hopefully complete theory for polygonal domains,  we imagined  a refined Green's function estimate  that involves both the distance to the vertex  and the distance to the boundary. This paper is about this improvement task. 

The main contents of this paper are as follows. In Section 2, we introduce a Green's function estimate of the time measurable parabolic operator  $\mathcal{L}=\frac{\partial}{\partial t}-\sum_{i,j=1}^d a_{ij}(t)D_{ij}$ defined on a conic domain $\mathcal{D} \subset \bR^d$ with a vertex at the origin. We prove an estimate of the type
 \begin{eqnarray}
 \nonumber
 G(t,s,x,y)&\leq&    N\, \frac{ e^{-\sigma  \frac{|x-y|^2}{t-s}}}{(t-s)^{d/2}}\,  \left(\frac{|x|}{\sqrt{t-s}}\wedge 1 \right)^{\beta_1}  \left( \frac{|y|}{\sqrt{t-s}}\wedge 1\right)^{\beta_2} \\
 && \times \left(\frac{\rho(x)}{\sqrt{t-s}}\wedge 1 \right)  \left( \frac{\rho(y)}{\sqrt{t-s}}\wedge 1\right), \quad \beta_1,\ \beta_2\ge 0,\label{eqn 10.31.1}
 \end{eqnarray}
 where $\rho(x):=\text{dist}(x,\partial \mathcal{D})$. The  ranges of $\beta_1$ and $\beta_2$ are determined by $\mathcal{D}$ and 
 $\mathcal{L}$ and described in Remark \ref{characterization}. Note that estimate \eqref{eqn 10.31.1}  involves both the distance to the vertex  and the distance to the boundary, and gives a subtle decay rate as $x,y$ approach  the boundary or  the origin. In Sections 3 and 4,  we obtain some critical upper bounds of $\beta_1, \beta_2$ for the operator $\cL$.

In this paper we use the following notations:
\begin{itemize}
\item[-] We use $:=$ to denote a definition.
\item[-] $\alpha\wedge \beta=\min\{\alpha,\beta\}$, $\alpha\vee \beta=\max\{\alpha,\beta\}$
\item[-] $N(\cdots)$  means  a constant depending only on what are indicated.
\item[-] $D_{ij}u=\frac{\partial^2 u}{\partial x_j\partial x_i}$
\end{itemize}
and
\begin{itemize}
\item[-]  $B_R(x)=\{y\in\mathbb{R}^d\mid |y-x|<R\}$
\item[-]  $B^{\mathcal{D}}_R(x)=B_R(x)\cap\mathcal{D}$
\item[-]  $Q_R(t,x)=(t-R^2,t]\times B_R(x)$
\item[-]  $Q^{\mathcal{D}}_R(t,x)=(t-R^2,t]\times (B_R(x)\cap \mathcal{D})$.
\end{itemize}

Also, we will frequently use the following sets of functions (see \cite{Kozlov Nazarov 2014}).
\begin{itemize}
\item[-]
$\mathcal{V}(Q_R(t_0,x_0))$ : the set of functions $u$ defined at least on $Q_R(t_0,x_0)$ and satisfying
\begin{equation*}
\sup_{t\in(t_0-R^2,t_0]}\|u(t,\cdot)\|_{L_2(B_{R}(x_0))} +\|\nabla u\|_{L_2(Q_{R}(t_0,x_0))}<\infty.
\end{equation*}
\item[-] 
$\mathcal{V}_{loc}(Q_R(t_0,x_0))$ : the set of functions $u$ defined at least on $Q_R(t_0,x_0)$ and satisfying
\begin{equation*}
u\in \mathcal{V}(Q_r(t_0,x_0)),  \quad \forall r\in (0,R).
\end{equation*}
\item[-]
$\mathcal{V}(Q^{\mathcal{D}}_R(t_0,x_0))$ : the set of functions $u$ defined at least on $Q^{\mathcal{D}}_R(t_0,x_0)$ and satisfying
\begin{equation*}
\sup_{t\in(t_0-R^2,t_0]}\|u(t,\cdot)\|_{L_2(B^{\mathcal{D}}_{R}(x_0))} +\|\nabla u\|_{L_2(Q^{\mathcal{D}}_{R}(t_0,x_0))}<\infty.
\end{equation*}
\item[-] 
$\mathcal{V}_{loc}(Q^{\mathcal{D}}_R(t_0,x_0))$ : the set of functions $u$ defined at least on $Q^{\mathcal{D}}_R(t_0,x_0)$ and satisfying
\begin{equation*}
u\in \mathcal{V}(Q^{\mathcal{D}}_r(t_0,x_0)),  \quad \forall r\in (0,R).
\end{equation*}
\end{itemize}

\mysection{Main result}
\label{sec:Main}

We define our conic domain  in $\mathbb{R}^d$ by
$$
\mathcal{D}=\Big\{x\in\mathbb{R}^d\setminus\{0\} \ \Big|  \ \ \frac{x}{|x|}\in \mathcal{M} \Big\},
$$
where $\mathcal{M}$ is a connected open subset in the sphere $S^{d-1}=\{\xi \in \mathbb{R}^d\mid \ |\xi|=1\}$ which has $C^{2}$ boundary.
Here, $C^2$ boundary means that for any fixed point $p\in S^{d-1}\setminus \overline{\mathcal{D}}$ and the stereographic projection of  $S^{d-1}\setminus\{p\}$ onto the tangent hyperplane at $-p$, the antipode of $p$, the image of $\mathcal{D}$ has $C^2$ boundary in the hyperplane. 
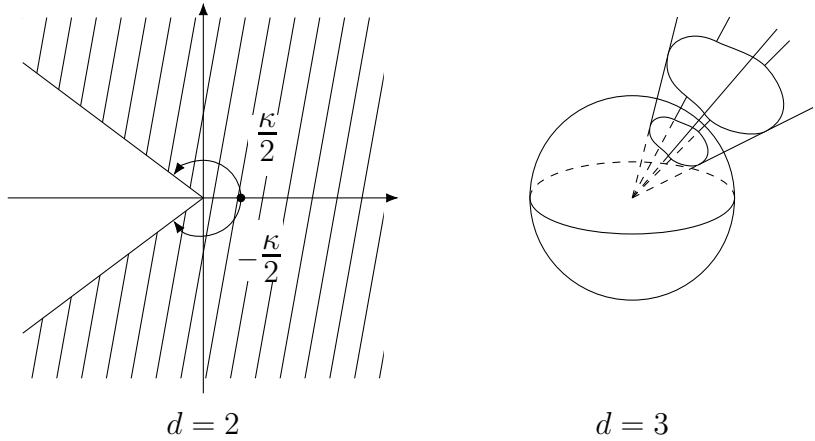
\begin{figure}[ht]
\begin{tikzpicture}[> = Latex]
\draw[->] (-2.6,0) -- (2.6,0);
\draw[->] (0,-2.6) -- (0,2.6);
\draw (0,0) -- (-2.4,1.8);
\draw (0,0) -- (-2.4,-1.8);
\draw[->] (0.5,0) arc(0:{180-atan(3/4)}:0.5);
\draw[<-] ({-180+atan(3/4)}:0.5) arc({-180+atan(3/4)}:0:0.5);
\begin{scope}
\clip (0,0)--(-2.4,-1.8)--(-2.4,-2.4)--(2.4,-2.4)--(2.4,2.4)--(-2.4,2.4)--(-2.4,1.8)--(0,0);
\foreach \i in {-4,-3.67,...,3}
{\draw (\i,-2.8)--(\i+1,2.8);}
\path[fill=white] (0.60,0.3) rectangle (1.05,1.1);
\path (0.85,0.35) node[above] {{\Large $\frac{\kappa}{2}$}};
\path[fill=white] (0.45,-0.3) rectangle (1.05,-1.1);
\path (0.75,-0.35) node[below] {{\small $-$}{\Large $\frac{\kappa}{2}$}};
\draw[fill=black] (0.5,0) circle (0.5mm);
\end{scope}
\path (0,-3) node {$d=2$};
\end{tikzpicture}
\begin{tikzpicture}[> = Latex]

\begin{scope}[scale=0.8]
\clip (0,0) circle (3.7) ;
\draw (0,0) circle (1.7);
\draw (1.7,0) arc(0:-180:1.7 and 0.6);
\draw[dashed] (1.7,0) arc(0:180:1.7 and 0.6);
\clip (-3,-3) -- (3,-3) -- (3,3) -- (-3,3);
\draw (1.1,0.55) -- (6,3);
\draw (0.3,1.2) -- (1.5,6);
\draw[dashed] (0,0) -- (1.1,0.55);
\draw[dashed] (0,0) -- (1.1,1.1);
\draw[dashed] (0,0) -- (0.7,1.3);
\draw[dashed] (0,0) -- (0.3,1.2);
\draw[dashed] (0,0) -- (0.6,0.7);
\draw (1.1,1.1) -- (3.3,3.3);
\draw (0.7,1.3) -- (2.1,3.9);


\draw (1.1,0.55) .. controls (1.4,0.7) and (1.2,1) .. (1.1,1.1) .. controls (1.0,1.2) and (0.8,1.25) .. (0.7,1.3) .. controls (0.5,1.4) and (0.33,1.32) .. (0.3,1.2) .. controls (0.23,0.92) and (0.5,0.85) .. (0.6,0.7) .. controls (0.7,0.55) and (1,0.5) .. (1.1,0.55);

\begin{scope}[scale=2]
\draw[fill=white] (1.1,0.55) .. controls (1.4,0.7) and (1.2,1) .. (1.1,1.1) .. controls (1.0,1.2) and (0.8,1.25) .. (0.7,1.3) .. controls (0.5,1.4) and (0.33,1.32) .. (0.3,1.2) .. controls (0.23,0.92) and (0.5,0.85) .. (0.6,0.7) .. controls (0.7,0.55) and (1,0.5) .. (1.1,0.55);
\end{scope}
\draw (0.6,0.7) -- (3.0,3.5);

\end{scope}
\path (0,-3) node {$d=3$};
\end{tikzpicture}
\caption{Cases of $d=2$ and $d=3$}
\end{figure}

For example, when $d=2$, for each fixed  angle $\kappa\in\left(0,2\pi\right)$ we can consider
\begin{equation}\label{wedge in 2d}
\mathcal{D}=\mathcal{D}_{\kappa}=\left\{(r\cos\theta,\ r\sin\theta)\in\mathbb{R}^2 \mid r\in(0,\ \infty),\ -\frac{\kappa}{2}<\theta<\frac{\kappa}{2}\right\}.
\end{equation}

In this paper we consider the Green's function of the operator  
\begin{equation}
   \label{eqn 10.28.2}
   \mathcal{L}=\frac{\partial}{\partial t}-\sum_{i,j}a_{ij}(t)D_{ij}
   \end{equation}
 with the domain $\mathcal{D}$.
We assume that  the diffusion coefficients $a_{ij}$, $i,j=1,\ldots,d$,  are real valued measurable functions of $t$, $a_{ij}=a_{ji}$, $i,j=1,\ldots,d$, and satisfy the uniform parabolicity condition, i.e. there exists a constant $\nu\in (0,1]$ such that for any $t\in\mathbb{R}$ and $\xi=(\xi_1,\ldots,\xi_d)\in\mathbb{R}^d$,
\begin{eqnarray}
\nu |\xi|^2\le \sum_{i,j}a_{ij}(t)\xi_i\xi_j\le \nu^{-1}|\xi|^2.  \label{uniform parabolicity}
\end{eqnarray}

 We denote  the Green's function  by $G(t,s,x,y)$.  By the definition of Green's function  $G$ is nonnegative and, for any fixed $s\in \mathbb{R}$ and $y\in\mathcal{D}$,  the function $v=G(\ \cdot,s,\ \cdot,y)$ satisfies
$$
\mathcal{L}v=0\quad\textrm{in}\quad (s,\infty)\times\mathcal{D}\; ; \quad v=0\quad \textrm{on}\quad (s,\infty)\times\mathcal{\partial D} \; ; \quad  v(t,\cdot)=0\quad \textrm{for} \quad t<s.
$$
Also, in this paper we use the notations $\rho_0(x)=|x|$, $\rho(x)={\text{dist}}(x,\partial \mathcal{D})$  and
$$
R_{t,x}:=\frac{|x|}{\sqrt{t}}\wedge 1=\frac{\rho_0(x)}{\sqrt{t}}\wedge 1, \quad J_{t,x}:=\frac{\rho(x)}{\sqrt{t}}\wedge 1.
$$
 \begin{remark}\label{another expression}
 Since $\frac{a}{a+1}\le a\wedge 1\le 2\cdot\frac{a}{a+1}$ for any $a\ge 0$,  we can also define  $R_{t,x}$ and $J_{t,x}$ by 
  $$
 R_{t,x}:=\frac{\rho_0(x)}{\rho_0(x)+\sqrt{t}},\quad J_{t,x}:=\frac{\rho(x)}{\rho(x)+\sqrt{t}}.
  $$
  \end{remark}

From the probabilitstic point of view related to  a Brownian motion killed at the boundary of $\partial\mathcal{D}$, $G$ is  essentially a transition probability   and bounded by a constant multiple of Gaussian density function: 
\begin{eqnarray}
 0\le G(t,s,x,y)\le N \frac{1}{(t-s)^{d/2}} e^{-\sigma\frac{|x-y|^2}{t-s}}, \quad t>s,\quad x,y\in\mathcal{D},\label{rough estimate}
\end{eqnarray}
where  the constants  $N$, $\sigma>0$ depend only on space dimension $d$ and $\nu$ in the assumption \eqref{uniform parabolicity}.

Having further information of the domain, the right hand side of \eqref{rough estimate} can be refined. Especially, for our conic domains $\mathcal{D}$, one can pursue  the following type of estimate
$$
 G(t,s,x,y)\leq N\frac{1}{(t-s)^{d/2}}\  R^{\lambda^+}_{t-s,x}\ R^{\lambda^-}_{t-s,y}\ e^{-\sigma \frac{|x-y|^2}{t-s}}, \quad t>s,\quad x,y\in\mathcal{D}
$$
for some positive constants  $\lambda^+,\lambda^-$. Since $R_{t,x}$ is less than equal to 1, this estimate is sharper as we find bigger $\lambda^+,\lambda^-$ satisfying the estimate.   

\begin{remark}\label{characterization}
As in \cite[Section 2]{Kozlov Nazarov 2014}, the critical upper bound $\lambda^+_c>0$ of $\lambda^+$ can be characterized by the supremum of all $\lambda$ such that for some  constant   $K_0=K_0(\cL,\cM,\lambda)$ it holds that 
\begin{equation}
      \label{eqn 7.29.1}
|u(t,x)|\le K_0 \left(\frac{|x|}{R}\right)^{\lambda}\sup_{Q^{\mathcal{D}}_{\frac{3R}{4}}(t_0,0)}\ |u|,
\quad
\forall \;(t,x)\in Q^{\mathcal{D}}_{R/2}(t_0,0)
\end{equation}
for any $t_0>0$, $R>0$, and $u$ belonging to  $\mathcal{V}_{loc}(Q^{\mathcal{D}}_R(t_0,0))$  and satisfying
\begin{equation*}
\mathcal{L}u=0\quad \text{in}\; Q^{\mathcal{D}}_R(t_0,0)\quad ; \;\quad
u(t,x)=0\quad\text{for}\;\; x\in\partial\mathcal{D}.
\end{equation*}
The value of $\lambda^{+}_c$ does not change if one replaces $\frac{3}{4}$ in \eqref{eqn 7.29.1} by any number in $(1/2,1)$ (see \cite[Lemma 2.2]{Kozlov Nazarov 2014}).

Moreover, the critical upper bound $\lambda^-_c>0$ of $\lambda^-$ is characterized by the supremum of  $\lambda$ with above property  for the operator
\begin{equation}
\hat{\mathcal{L}}=\frac{\partial}{\partial t}-\sum_{i,j}a_{ij}(-t)D_{ij}.\label{another operator}
\end{equation}

Both  $\lambda^+_c$  and  $\lambda^-_c$ will definitely depend on $\mathcal{M}=\mathcal{D}\cap S^{d-1}$. Especially when $\mathcal{D}=\mathcal{D}_{\kappa}$ in \eqref{wedge in 2d}, $\lambda^+_c$  and  $\lambda^-_c$ will depend on  the opening angle $\kappa$. If in addition $\mathcal{L}$ is the heat opeartor, $\mathcal{L}=\frac{\partial}{\partial t}-\Delta_x$, then
$$
\lambda^+_c=\lambda^-_c=\frac{\pi}{\kappa}.
$$  
See Section 2 of \cite{Kozlov Nazarov 2014} and Section \ref{sec:critical upper bounds} of this paper for details. 
\end{remark}
The following lemma is, we think, the most updated estimate of $G$  among the ones involving  $R_{t,x}$ only.

 \begin{lemma}
                                                                    \label{Koz Naz}
   Let   $\lambda^+\in(0,\lambda^+_c)$, $\lambda^-\in(0,\lambda^-_c)$, and  denote   $K^+_0:=K_0(\cL,\cM,\lambda^+)$ and  $K^-_0:=K_0(\hat{\cL},\cM,\lambda^-)$.
   Then there exist positive constants 
   $N=N(\mathcal{M}, \nu,  \lambda^{\pm}, K^{\pm}_0)$ and $\sigma=\sigma(\nu) $  such that
 \begin{equation}
  G(t,s,x,y)\leq   \frac{N}{(t-s)^{d/2}} R^{\lambda^+}_{t-s,x}\ R^{\lambda^-}_{t-s,y}\ e^{-\sigma \frac{|x-y|^2}{t-s}} \label{less rough estimate}
 \end{equation}
 and
  \begin{equation*}
  |\nabla_x G(t,s,x,y)|\leq   \frac{N}{(t-s)^{(d+1)/2}} R^{\lambda^+-1}_{t-s,x}\ R^{\lambda^-}_{t-s,y}\ e^{-\sigma \frac{|x-y|^2}{t-s}}
 \end{equation*}
   for any $t>s$, $x,y\in\mathcal{D}$.  
  \end{lemma}
  
  \begin{proof}
  See \cite[Theorem 3.10]{Kozlov Nazarov 2014}. We only remark that in \cite{Kozlov Nazarov 2014} the dependency of $N$ on $K_0^{\pm}$ is taken for granted and omitted. 
  By inspecting the proof of  \cite[Theorem 3.10]{Kozlov Nazarov 2014} one can  check that constant $N$ actually depends also on $K_0^{\pm}$.
  \end{proof}

 \begin{remark}
  In fact, \cite{Kozlov Nazarov 2014}  has the estimates of the derivatives of $G$ up to the second order that contain Lemma \ref{Koz Naz} as a part.  We refer to Theorem 3.10 of  \cite{Kozlov Nazarov 2014}.
 Yet,  the estimates involve  $R_{t,x}$ only.
 \end{remark}
 
\begin{remark}
Despite  the beauty in  estimate \eqref{less rough estimate},  we note that  the right hand side of \eqref{less rough estimate}  does not go to zero as $x$ or $y$ approaches  boundary  of $\mathcal{D}$, meaning that  the estimate is not sharp enough in terms of the boundary behavior of the Green's function. 

On the other hand, for any domain  satisfying, for instance, the uniform exterior ball condition, the corresponding Green's function of $\mathcal{L}$  is bounded by the constant multiple of
\begin{equation}
\frac{1}{(t-s)^{d/2}} J_{t-s,x}\ J_{t-s,y}\ e^{-\sigma \frac{|x-y|^2}{t-s}},\nonumber
\end{equation}
which is now forcing the degeneracy of the Green's function at the boundary (see e.g. \cite{Cho 2006}).
\end{remark}

Of course, our domains, for instance,  like $\mathcal{D}_{\kappa}$ in \eqref{wedge in 2d} does not satisfy  the uniform exterior ball condition if $\kappa>\pi$. However, for any $\kappa$, $\mathcal{D}_{\kappa}$ is mostly flat except  a samll neighborhood of the vertex  and  we hoped  a refined estimate that involves both  $R_{t,x}$ and $J_{t,x}$ together. After all, we settled down with the following theorem, which  is the refined estimate we mentioned in the  introduction and is the main result of this paper.

 \begin{thm}\label{key}
   Let   $\lambda^+\in(0,\lambda^+_c)$, $\lambda^-\in(0,\lambda^-_c)$, and  denote  $K^+_0:=K_0(\cL,\cM,\lambda^+)$ and  $K^-_0:=K_0(\hat{\cL},\cM,\lambda^-)$.
   Then there exist positive constants 
   $N=N(\mathcal{M}, \nu,  \lambda^{\pm}, K^{\pm}_0)$ and $\sigma=\sigma(\nu)$  such that
\begin{equation}\label{key estimate}
  G(t,s,x,y)\leq \frac{N}{(t-s)^{d/2}}\,  R^{\lambda^+ -1}_{t-s,x}\, R^{\lambda^- -1}_{t-s,y}\,  J_{t-s,x} \, J_{t-s,y}\, e^{-\sigma \frac{|x-y|^2}{t-s}}
\end{equation}
   for any $t>s$, $x,y\in\mathcal{D}$.  
  \end{thm}
  \begin{remark}
  Obviously estimate \eqref{key estimate}  is sharper than  estimate \eqref{less rough estimate} since $J_{t,x}\le R_{t,x}$. Moreover, estimate \eqref{key estimate} gives delicate boundary behavior of Green's funciton.
 \end{remark}
 \begin{remark}
 The strategy of our proof of Theorem \ref{key estimate} is inspired by \cite{Cho 2006} and \cite{Riahi 2005} although the details are quite different.
 \end{remark}
 
In the proof of  Theorem \ref{key}, we will use the following two lemmas from \cite{Kozlov Nazarov 2014}.
\begin{lemma}[Proposition 3.2 of \cite{Kozlov Nazarov 2014}]\label{estimate by sup 1}
Let $u$ belong to  $\mathcal{V}(Q_R(t_0,x_0))$  and satisfy
$\mathcal{L}u=0$
in $Q_R(t_0,x_0)$,
then 
$$
|\nabla u(t,x)|\le  \frac{N}{R} \sup_{Q_R(t_0,x_0)} |u|,\quad  \forall (t,x)\in Q_{R/2}(t_0,x_0),
$$
 where the constant $N$ depends only on $\nu$ and $d$.
\end{lemma}
  \begin{lemma}[Proposition 3.4 of \cite{Kozlov Nazarov 2014}]\label{estimate by sup 2}
There exists a sufficently samll $\delta_0$ such that the following holds for any $\delta\in (0,\delta_0)$ :
 Let  $x_0\in \mathcal{D}$, $\rho(x_0)<\delta|x_0|$, and $R\le\frac{|x_0|}{2}$. 
Then if  $u$ belongs to  $\mathcal{V}(Q^{\mathcal{D}}_R(t_0,x_0))$  and satisfies
$\mathcal{L}u=0$
in $Q^{\mathcal{D}}_R(t_0,x_0)$ and $u(t,x)=0$ for $x\in\partial\mathcal{D}$,
then 
$$
|\nabla u(t,x)|\le  \frac{N}{R} \sup_{Q^{\mathcal{D}}_R(t_0,x_0)} |u|,\quad  \forall (t,x)\in Q^{\mathcal{D}}_{R/8}(t_0,x_0),
$$
 where the constant $N$ depends only on $\mathcal{M},\nu,\delta$.
\end{lemma}
 
 \begin{proof}[Proof of Theorem \ref{key}]
 \vspace{0.1cm}
 \quad

  {\textbf{1}}. First, we fix $s\in \mathbb{R}$, $y\in\mathcal{D}$.  We show that  there exist positive constants  $N=N(\mathcal{M}, \nu,  \lambda^{\pm}, K^{\pm}_0)$ and $\sigma=\sigma(\nu)$ such that   for any $t\in(s,\infty)$ and $x\in\mathcal{D}$,
 \begin{equation}
  G(t,s, x, y)\leq \frac{N}{(t-s)^{d/2}}\ J_{t-s,x}\  R^{\lambda ^+-1}_{t-s,x}\  R^{\lambda^-}_{t-s,y} \  e^{-\sigma \frac{|x-y|^2}{t-s}}.\label{second estimate}
 \end{equation}

For given $t\in(s,\infty)$,  we consider the following two cases of  $x\in\mathcal{D}$.
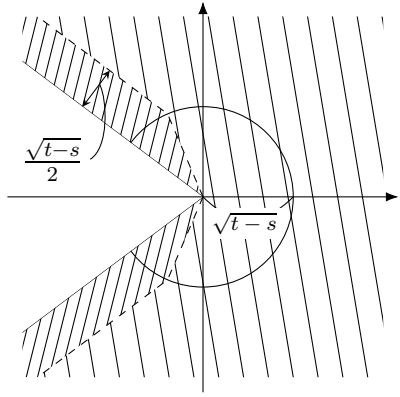
\begin{figure}[ht]
\centering
\begin{tikzpicture}
\draw[->,> = Latex] (-2.6,0) -- (2.6,0);
\draw[->,> = Latex] (0,-2.6) -- (0,2.6);
\begin{scope}
\clip (0,0)--(-2.4,1.8)--(-2.4,2.4)--(2.4,2.4)--(2.4,-2.4)--(-2.4,-2.4)--(-2.4,-1.8)--cycle;
\path ($0.5*({atan(4/3)}:1.2)$) coordinate (A1);
\path ($0.5*({-atan(4/3)}:1.2)$) coordinate (A2);
\draw (0,0) -- (-3.2,2.4);
\draw[dashed] (0,0) -- ($(0,0)!2!-60:(A2)$) -- ($(A2)+(-3.2,-2.4)$);
\draw (0,0) -- (-3.2,-2.4);
\draw[dashed] (0,0) -- ($(0,0)!2!60:(A1)$) -- ($(A1)+(-3.2,2.4)$);
\draw (0,0) circle (1.2) ;
\begin{scope}
\clip (-3.2,-2.4)--($(-3.2,-2.4)+(A2)$) -- ($(0,0)!2!-60:(A2)$)--(0,0)--cycle;
\foreach \i in {-4,-3.85,...,0}
{\draw (\i,-2.4)--(\i+1.2,2.4);}
\end{scope}
\begin{scope}
\clip (-3.2,2.4)--($(-3.2,2.4)+(A1)$) -- ($(0,0)!2!60:(A1)$)--(0,0)--cycle;
\foreach \i in {-4,-3.85,...,0}
{\draw (\i,-2.4)--(\i+1.2,2.4);}
\end{scope}
\begin{scope}
\clip (0,0) -- ($(0,0)!2!60:(A1)$) -- ($(A1)+(-3.2,2.4)$) -- (2.4,2.4) -- (2.4,-2.4) -- ($(A2)+(-3.2,-2.4)$) -- ($(0,0)!2!-60:(A2)$) -- (0,0);
\foreach \i in {-3,-2.67,...,4}
{\draw (\i,3)--(\i+1,-3);}
\end{scope}
\draw (0,0)..controls (0.4,-0.4)and(0.8,-0.4).. (1.2,0);
\path[fill=white] (0.15,-0.55) rectangle (1.05,-0.15);
\path (0.55,-0.35) node {{\fontsize{8pt}{0mm}\selectfont $\sqrt{t-s}$ }};
\end{scope}
\draw[<->, >=stealth, thin] (-1.6,1.2) -- ($(-1.6,1.2)+(A1)$);
\path[fill=white] (-2.5,1) rectangle (-1.5,0.5);
\path (-2,0.5) node {{\fontsize{12pt}{0mm}\selectfont $\frac{\sqrt{t-s}}{2}$}};
\draw ($(-1.6,1.2)+0.25*({atan(4/3)}:1.5)$) .. controls ($(-1.6,1.2)+0.5*(A1)+(0.15,-0.15)$)and(-1.25,0.65)..(-1.5,0.5);
\end{tikzpicture}
\caption{Two cases of $x$}
\end{figure}

   -  {\bf{Case}}
  $\rho(x)\geq\frac{1}{2}\left(|x|\wedge \sqrt{t-s}\right)$.
 
  In this case, by assumption we have
  $$
 2\frac{\rho(x)}{\sqrt{t-s}}\geq\left(\frac{|x|}{\sqrt{t-s}}\wedge 1\right).
  $$
  Therefore,
  \begin{equation}\label{20191028-1}
  R_{t-s,x}=\frac{|x|}{\sqrt{t-s}}\wedge 1\leq 2 \frac{\rho(x)}{\sqrt{t-s}}\wedge 2=2\left(\frac{\rho(x)}{\sqrt{t-s}}\wedge 1\right).
  \end{equation}
  Then, using Lemma \ref{Koz Naz}, we immediately get  \eqref{second estimate}.
   
 - {\bf{Case}} 
  $\rho(x)<\frac{1}{2}\left(|x|\wedge \sqrt{t-s}\right)$; the point close to the boundary.

  For such point $x\in\mathcal{D}$, there exists $x_0\in\partial \mathcal{D}$ such that $|x-x_0|=\rho(x)$. For this $x_0\in\partial \mathcal{D}$, $G(t,s,x_0,y)=0$ and  there exists $\theta\in (0,1)$ such that 
\begin{eqnarray}
G(t,s,x,y)&=&G(t,s,x,y)-G(t,s,x_0,y)\nonumber\\
&\leq& |x-x_0| |\nabla_x G(t,s,\bar{x},y)|\nonumber\\
&=&\rho(x)|\nabla_x G(t,s,\bar{x},y)|,\label{MVT}
\end{eqnarray}
where $\bar{x}=(1-\theta)x+\theta x_0\in\mathcal{D}$.

 To estimate the gradient part, we make use of  Lemma \ref{Koz Naz}. 
Now, since 
$$
|\bar{x}|\geq |x|-\theta |x-x_0|\geq |x|- \rho(x)>\frac{1}{2}|x|,\quad
\ |\bar{x}|\leq|x|+\theta|x-x_0|\leq |x|+\rho(x) < 2|x|,
$$
we note that
$$
\frac{1}{2}R_{t-s,x}\leq R_{t-s,\bar{x}}\leq 2 R_{t-s,x}.
$$
 In addition, the inequalities
\begin{equation*}
|x-y|\leq |\bar{x}-y|+|\bar{x}-x|\leq |\bar{x}-y|+|x-x_0|\leq |\bar{x}-y|+\sqrt{t-s}
\end{equation*}
give
$$
 -|\bar{x}-y|^2\leq -\frac{1}{2}|x-y|^2+t-s.
$$
Hence, $|\nabla_x G(t,s,\bar{x},y)|$ is bounded by
$$
N' \frac{1}{(t-s)^{(d+1)/2}} R_{t-s,x}^{\lambda^+-1}R_{t-s,y}^{\lambda^-} e^{-\sigma'\frac{|x-y|^2}{t-s}},
$$
where  $N'=N'(\mathcal{M}, \nu,  \lambda^{\pm}, K^{\pm}_0)>0$ and $\sigma'=\sigma'(\nu)>0$. This, \eqref{MVT}, and $\rho(x)\le \sqrt{t-s}$  lead us to  \eqref{second estimate} again.
\vspace{0.2cm}
\quad

  {\textbf{2}}.  Now, we consider the operator $\hat{\mathcal{L}}$ defined in \eqref{another operator}. Let $\hat{G}$ denote the Green's function for $\hat{\mathcal{L}}$ with the same domain $\mathcal{D}$. Note that the diffusion coefficients $a_{ij}(-t)$, $i,j=1,\ldots,d$, also satisfy the uniform parabolicity condition \eqref{uniform parabolicity} with the same $\nu$.
Since for any  $s\in\mathbb{R}$ and $y\in \mathcal{D}$, $\hat{\mathcal{L}}\hat{G}(\cdot,s,\cdot,y)=0$ on $(s,\infty)\times \mathcal{D}$ and $\hat{G}(\cdot,s, \cdot,y)=0$ on $(s,\infty)\times\partial\mathcal{D}$, we can repeat the argument in Step 1 literally line by line. Hence, denoting  the critical upper bounds of $\lambda$ for the operator $\hat{\mathcal{L}}$ by $\hat\lambda^+_c$, $\hat\lambda^-_c$ and noting that $\hat\lambda^+_c=\lambda^-_c$, $\hat\lambda^-_c=\lambda^+_c$ by Remark \ref{characterization}, with the same constants  $N,\sigma$ in \eqref{second estimate},  we obtain that
 \begin{equation}
  \hat{G}(t,s, x, y)\leq \frac{N}{(t-s)^{d/2}}\ J_{t-s,x}\  R^{\lambda ^--1}_{t-s,x}\  R^{\lambda^+}_{t-s,y} \  e^{-\sigma \frac{|x-y|^2}{t-s}}\label{another0 second estimate}
 \end{equation}
for any $t>s$ and $x,y\in\mathcal{D}$. Note that the locations of $\lambda^+$, $\lambda^-$ in \eqref{another0 second estimate} in comparison with the locations of them in \eqref{second estimate}. This is simply because $\lambda^-\in(0,\hat\lambda^+_c)$ and $\lambda^+\in(0,\hat\lambda^-_c)$.
 \vspace{0.2cm}
\quad

  {\textbf{3}}.  Next, using the result of Step 2 and the following identity 
 $$
 G(-s,-t,y,x)=\hat{G}(t,s,x,y)\  \ \text{or}\ \  G(t,s,x,y)=\hat{G}(-s,-t,y,x),\quad t>s
 $$
 which is due to a duality argument (see (3.12)  of \cite{Kozlov  Nazarov 2014} for the detail), we observe that with the same constants  $N,\sigma$ in \eqref{second estimate} we have
 \begin{eqnarray}
  G(t,s, x, y)&\leq& \frac{N}{(t-s)^{d/2}}\ J_{t-s,y}\  R^{\lambda ^--1}_{t-s,y}\  R^{\lambda^+}_{t-s,x} \  e^{-\sigma \frac{|x-y|^2}{t-s}}\nonumber\\
 &=& \frac{N}{(t-s)^{d/2}}\  R^{\lambda^+}_{t-s,x}\ J_{t-s,y}\  R^{\lambda ^--1}_{t-s,y} \  e^{-\sigma \frac{|x-y|^2}{t-s}}\label{another second estimate}
 \end{eqnarray}
 for any $t>s$ and $x,y\in\mathcal{D}$. 
 \vspace{0.2cm}
\quad

  {\textbf{4}}.   Finally to finish the proof of \eqref{key estimate} we repeat the argument in Step 1. 
  
   For the points $x$ away from the boundary  the argument is the same.  Indeed, if $\rho(x)\geq\frac{1}{2}\left(|x|\wedge \sqrt{t-s}\right)$, then  \eqref{20191028-1} and \eqref{another second estimate}  certainly give  \eqref{key estimate}. 
  
  Therefore, for the rest of the proof, we may assume  
$$
\rho(x)<\frac{1}{2}\left(|x|\wedge \sqrt{t-s}\right).
$$

 In this case we first show 
 \begin{equation}\label{third estimate}
 |\nabla_x G(t,s,x,y)|\leq N \frac{1}{(t-s)^{(d+1)/2}} J_{t-s,y} R_{t-s,x}^{\lambda^+-1}R_{t-s,y}^{\lambda^--1}e^{-\sigma\frac{|x-y|^2}{t-s}}.
 \end{equation}
 For this,  we fix $(s,y)$ and  set 
 $$
 u(t,x)=G(t,s,x,y).
 $$
Take $\delta\in(0,\delta_0\wedge 1/2)$, where $\delta_0$ is from  Lemma \ref{estimate by sup 2} which depends only on $\mathcal{M}$. We consider the following two cases.

 - {\bf{Case}}  $\rho(x)\geq \delta |x|$.
Put $R=\frac{\delta}{2}(|x|\wedge\sqrt{t-s})$ which is less than $\frac{1}{2}\rho(x)$ so that  $\bar{B}_{R}(x)\subset \mathcal{D}$. Since $u$ belongs to  $\mathcal{V}(Q_R(t,x))$
and satisfies
$\mathcal{L}u=0$  in  $Q_{R}(t,x)$, by Lemma \ref{estimate by sup 1}, we get
$$
|\nabla_xu(t,x)|\leq \frac{N}{R} \sup_{\substack{Q_{R}(t,x)}}|u|.
$$
We note that for $(r,z)\in Q_R(t,x)$, 
$$
0\leq t-r \leq\frac{t-s}{4},\quad \frac{3}{4}(t-s)\leq r-s\leq t-s,
$$
$$
|z|\leq|x|+R\leq 2|x|,\quad |z|\geq |x|-R\geq \frac{1}{2}|x|
$$
and

\begin{align*}
&|z-y| \geq |x-y|-R \geq |x-y|-\sqrt{t-s},\\
&-|z-y|^2 \leq -\frac{1}{2}|x-y|^2+(t-s), \\
& -\frac{|z-y|^2}{r-s}\leq -\frac{1}{2}\frac{|x-y|^2}{t-s}+\frac{4}{3}.
\end{align*}

Hence,  using \eqref{another second estimate} we get
\begin{eqnarray*}
|u(r,z)|&\leq&  \frac{N}{(r-s)^{d/2}}\  R^{\lambda^+}_{r-s,z}\ J_{r-s,y}\  R^{\lambda ^--1}_{r-s,y} \  e^{-\sigma \frac{|z-y|^2}{r-s}}\\
&\leq&
\frac{N}{(t-s)^{d/2}} R_{t-s,x}^{\lambda^+} J_{t-s,y}R_{t-s,y}^{\lambda^--1}e^{-\sigma'\frac{|x-y|^2}{t-s}}.
\end{eqnarray*}

Consequently, we have
\begin{align*}
|\nabla_x u(t,x)|&\leq \frac{N}{R}\sup_{\substack{Q_{R}(t,x)}}|u|\\
&\leq \frac{N}{|x|\wedge\sqrt{t-s}}\frac{1}{(t-s)^{d/2}}R_{t-s,x} R_{t-s,x}^{\lambda^+-1} J_{t-s,y}R_{t-s,y}^{\lambda^--1}e^{-\sigma'\frac{|x-y|^2}{t-s}}\\
&=\frac{N}{(t-s)^{(d+1)/2}}J_{t-s,y}R_{t-s,x}^{\lambda^+-1}R_{t-s,y}^{\lambda^--1}e^{-\sigma'\frac{|x-y|^2}{t-s}},
\end{align*}
and thus  \eqref{third estimate} is proved.

 - {\bf{Case}} $\rho(x)\leq \delta |x|$. In this case, we put $R=\frac{1}{2}(|x|\wedge \sqrt{t-s})$. Since $u$ belongs to $\mathcal{V}(Q^{\mathcal{D}}_R(t,x))$ and satisfies
$\mathcal{L}u=0$  in  $Q^{\mathcal{D}}_{R}(t,x)$, and $u(t,x)=0$ for $x\in\partial\mathcal{D}$, we can apply Lemma \ref{estimate by sup 2}, and have
 $$
 |\nabla_x u(t,x)|\leq \frac{N}{R} \sup_{\substack{Q^{\mathcal{D}}_R(t,x)}}|u|.
 $$
Similarly as before, we again  obtain  \eqref{third estimate}. 

Finally, by \eqref{MVT}, the computations below \eqref{MVT}, and \eqref{third estimate}, we obtain  \eqref{key estimate}. This ends the proof.
 \end{proof}

\mysection{On the critical upper bounds  $\lambda^{\pm}_c$}
         \label{sec:critical upper bounds}
In this section we discuss  some detailed informations of the critical upper bounds $\lambda^+_c$ and $\lambda^-_c$,  whose  characterizations are given  in Remark \ref{characterization}.

We first introduce some known results on $\lambda^{\pm}_c$. 
 The following statements are  the 3rd, the 8th, and the 7th  in Theorem 2.4 of \cite{Kozlov Nazarov 2014}:
\begin{itemize}
\item If $\mathcal{L}=\cL_0:=\frac{\partial}{\partial t}-\Delta_x$,  then
\begin{equation}\label{CUB1}
\lambda^{\pm}_c(\cL_0, \cD)=-\frac{d-2}{2}+\sqrt{\Lambda+\frac{(d-2)^2}{4}},
\end{equation}
where $\Lambda$ is the first eigenvalue of Laplace-Beltrami operator with the Dirichlet condition on domain $\mathcal{M}=\mathcal{D}\cap S^{d-1}$, where $S^{d-1}$ is the  sphere with radius 1 in $\mathbb{R}^d$.

\item 

Suppose that $(a_{ij})_{d\times d}$ is a constant matrix.  Then
\begin{align}\label{CUB4}
\lambda^{\pm}_c(\cL,\cD)=\lambda^{\pm}_c(\cL_0, \widetilde{\cD})=-\frac{d-2}{2}+\sqrt{\widetilde{\Lambda}+\frac{(d-2)^2}{4}},
\end{align}
where $\widetilde{\Lambda}$ is the first eigenvalue of the Dirichlet boundary value problem to Beltrami-Laplacian in the domain 
$\widetilde{\cM}=\widetilde{\cD}\cap S^{d-1}$ while cone $\widetilde{\cD}$ is the image of $\cD$ under the change of variables $x\to y$ that reduces $(a_{ij})_{d\times d}$ to the canonical form $(\delta_{ij})_{d\times d}$ with the Kronecker delta $\delta_{ij}$, $i,j=1,\ldots,d$.

\item For the general operator $\mathcal{L}=\frac{\partial}{\partial t}-\sum_{i,j=1}^d a_{ij}(t)D_{ij}$ in \eqref{eqn 10.28.2}, we have
\begin{equation}\label{CUB2}
\lambda^{\pm}_c \geq -\frac{d}{2}+\nu\sqrt{\Lambda+\frac{(d-2)^2}{4}},
\end{equation}
where $\nu$  is  the uniform parabolicity constant in \eqref{uniform parabolicity}. 

\end{itemize}

 \begin{remark}\label{improvement on critical lambda}
 One big difference between  \eqref{CUB4} and \eqref{CUB2} is that  $``d"$ appears in   \eqref{CUB2} in place of $``d-2"$.  This actually causes a big gap between  \eqref{CUB4} and \eqref{CUB2}. To demonstrate this,  let $d=2$, $\mathcal{D}=\mathcal{D}_{\kappa}$ in \eqref{wedge in 2d}, and $\cL=\cL_0=\frac{\partial}{\partial t}-(D_{x_1x_1}+D_{x_2x_2})$. Then we can easily find $\Lambda$  in \eqref{CUB1}, which is the same as $\tilde{\Lambda}$ in \eqref{CUB4}. To find $\Lambda$,  we just need to find the smallest eigenvalue $\lambda>0$ and its eigenfunction $\phi=\phi(\theta)$  satisfying
$$
- \phi''=\lambda \phi,\quad -\frac{\kappa}{2}<\theta<\frac{\kappa}{2},\quad ;\quad \phi\left(\frac{\kappa}{2}\right)=\phi\left(-\frac{\kappa}{2}\right)=0,
$$
which  yields $\phi(\theta)=\cos(\sqrt{\lambda}\theta)$ and $\cos\left(\sqrt{\lambda}\ \kappa/2\right)=0$. Hence, the  eigenvalues  satisfy $\sqrt{\lambda}\ \kappa/2=\pi/2+k\pi$, $k=0,1,2,\ldots$,  and thus $\Lambda=\pi^2/\kappa^2$. 

In this example, if for instance  $\kappa=\pi$, then  \eqref{CUB2}  yields, as we can take $\nu=1$,  a trivial information $\lambda^{\pm}_c\geq 0$, whereas \eqref{CUB4} gives $\lambda^{\pm}_c=1$. 
\end{remark}

In this section we improve  \eqref{CUB2}. In particular, we will replace $d$ in \eqref{CUB2} by $d-2$. We assume that  the coefficients $a_{ij}(t)$, $i,j=1,\cdots,d,$ satisfy $a_{ij}(t)=a_{ji}(t)$,  and   there exist  constants $\nu_1, \nu_2>0$ such that for any $t\in\mathbb{R}$ and $\xi\in\mathbb{R}^d$,
\begin{eqnarray}
\nu_1 |\xi|^2\le \sum_{i,j}a_{ij}(t)\xi_i\xi_j\le \nu_2|\xi|^2. \label{uniform parabolicity2}
\end{eqnarray}
 The condition \eqref{uniform parabolicity} is a special case of this condition: $\nu_1=\nu, \nu_2=\nu^{-1}$.

\begin{thm}\label{CUB theorem}
Let  $\nu_1,\ \nu_2$  be  the uniform parabolicity constants in \eqref{uniform parabolicity2}. If
$$
\lambda<-\frac{d-2}{2}+\sqrt{\frac{\nu_1}{\nu_2}}\sqrt{\Lambda+\frac{(d-2)^2}{4}},
$$
then there exists a positive constant $K_0=K_0(\nu_1,\nu_2,\cM,\lambda)$ such that
\begin{equation*}
|u(t,x)|\le K_0 \left(\frac{|x|}{R}\right)^{\lambda}\sup_{Q^{\mathcal{D}}_{\frac78R}(t_0,0)}\ |u|,
\quad
\forall \;(t,x)\in Q^{\mathcal{D}}_{R/2}(t_0,0)
\end{equation*}
for any $t_0>0$, $R>0$, and $u$ belonging to  $\mathcal{V}_{loc}(Q^{\mathcal{D}}_R(t_0,0))$  and satisfying
\begin{equation*}
\mathcal{L}u=0\quad \text{in}\; Q^{\mathcal{D}}_R(t_0,0)\quad ; \;\quad
u(t,x)=0\quad\text{for}\;\; x\in\partial\mathcal{D}.
\end{equation*}
In particular, we have
\begin{equation}\label{CUB3}
\lambda^{\pm}_c \geq -\frac{d-2}{2}+\sqrt{\frac{\nu_1}{\nu_2}}\sqrt{\Lambda+\frac{(d-2)^2}{4}}.
\end{equation}
\end{thm}

Note that if $\nu\le \nu_1\le \nu_2\le \nu^{-1}$,  the right hand side of \eqref{CUB3} is quite bigger that that of \eqref{CUB3}. Indeed,
\begin{eqnarray*}
&&\left(-\frac{d-2}{2}+\sqrt{\frac{\nu_1}{\nu_2}}\sqrt{\Lambda+\frac{(d-2)^2}{4}}\right)-\left(-\frac{d}{2}+\nu\sqrt{\Lambda+\frac{(d-2)^2}{4}}\right)\\
&&=1+\Big(\sqrt{\frac{\nu_1}{\nu_2}}-\nu\Big)\sqrt{\Lambda+\frac{(d-2)^2}{4}} \geq 1.
\end{eqnarray*}

To prove the above theorem, we start with the following lemma which is a slight modificaiton of Lemma A.1 of \cite{Kozlov Nazarov 2014}.

\begin{lemma}\label{A}
Let $\mu^2<\frac{\nu_1}{\nu_2}\left(\Lambda+\frac{(d-2)^2}{4}\right)$ and   $0<\epsilon_1<\epsilon_2\le 1$.  Then 
there exists a constant $N$  depending only on $\mu,\epsilon_1,\epsilon_2, \nu_1, \nu_2, \cM$ such that
$$
\int_{Q^{\mathcal{D}}_{\epsilon_1 R}(t_0,0)} |x|^{2\mu} |\nabla u|^2 dxdt+\int_{Q^{\mathcal{D}}_{\epsilon_1 R}(t_0,0)} |x|^{2\mu-2} |u|^2 dxdt \le N  R^{2\mu-2} \int_{Q^{\mathcal{D}}_{\epsilon_2 R}(t_0,0)} | u|^2 dxdt
$$
  for any  $R>0$ and any function $u$ belonging to $\mathcal{V}_{loc}(Q^{\mathcal{D}}_R(t_0,0))$ and satisfying 
$\mathcal{L}u=0$  in  $Q^{\mathcal{D}}_R(t_0,0)$, $u=0$ on $\mathbb{R}\times \partial\mathcal{D}$.
\end{lemma}
\begin{proof}
The proof of this lemma is  almost the  same as that of Lemma A.1 of \cite{Kozlov Nazarov 2014}.  The only difference is that we use conditon \eqref{uniform parabolicity2} instead of  condition \eqref{uniform parabolicity}.
\end{proof}

\begin{proof}[Proof of Theorem \ref{CUB theorem}]

\quad

{\textbf{1}}. Refering to Remark \ref{characterization}, we note that  it is enough to show that  for any $\mu\in\mathbb{R}$  satisfying
$\mu^2<\frac{\nu_1}{\nu_2}\left(\Lambda+\frac{(d-2)^2}{4}\right)$, there exists a constant $N$ depending only on $\mathcal{M},\mu,\nu_1, \nu_2$ such that
\begin{equation}
    \label{spde}
|u(t,x)|\le N \left(\frac{|x|}{R}\right)^{-\frac{d-2}{2}+\mu}\sup_{Q^{\mathcal{D}}_{\frac78R}(t_0,0)}\ |u|,
\quad
\forall \;(t,x)\in Q^{\mathcal{D}}_{R/2}(t_0,0)
\end{equation}
  for any $t_0>0$, $R>0$, and $u$  belonging to $\mathcal{V}_{loc}(Q^{\mathcal{D}}_R(t_0,0))$ and satisfying 
\begin{equation*}
\mathcal{L}u=0\quad \text{in}\; Q^{\mathcal{D}}_R(t_0,0)\quad ; \;\quad
u(t,x)=0\quad\text{for}\;\; x\in\partial\mathcal{D}.
\end{equation*}
Also, we note that we may assume $t_0=0$, $R=1$.

{\textbf{2}}. Take any function $u$ satisfying the conditions in Step 1 with $t_0=0,\ R=1$ and take any $(t,x)\in Q^{\mathcal{D}}_{1/2}(0,0)$. Let us denote  
$$r=|x| \,\Big(<\frac12\,\Big), \quad D_{r}=(t-r^2/4,t]\times (B_{\frac{3}{2}r}(0)\setminus B_{\frac12 r}(0)).
$$ 
Then as in the proof of statement 7 of Theorem 2.4 in \cite{Kozlov Nazarov 2014}, we have  
\begin{eqnarray}
|u(t,x)|^2&\le& N r^{-d-2}\int_{D_{r}} |u(\tau,y)|^2dyd\tau\nonumber\\
&\le&N r^{-d+2\mu}\int_{D_{r}}|y|^{-2\mu-2} |u(\tau,y)|^2dyd\tau.\label{local}
\end{eqnarray}
The last inequality in \eqref{local} holds since for the points $y$ in $D_{r}$, $|y|$ are comparable with $r$.

Now, we define a time-changed function of $u$: 
$$
v(s,y):=u(t+r^2 s,y).
$$
This function is well defined at least on $Q^{\mathcal{D}}_1(0,0)$  due to  $t+r^2 s\in (-1,0]$ for $s\in (-1,0]$.
Moreover, $v$ belongs to $ \mathcal{V}_{loc}(Q^{\mathcal{D}}_1(0,0)) $
and satisfies
\begin{equation*}
\tilde{\mathcal{L}}v=0\quad \text{in}\; Q^{\mathcal{D}}_1(0,0)\quad ; \;\quad
v=0\quad\text{on}\;\; \mathbb{R}\times \partial\mathcal{D},
\end{equation*}
where  $\tilde{\mathcal{L}}=\frac{\partial}{\partial s}-\sum_{i,j}r^2 a_{ij}(s)D_{ij}$. We note that 
\begin{eqnarray*}
r^2\nu_1 |\xi|^2\le \sum_{i,j}r^2a_{ij}(s)\xi_i\xi_j\le r^2\nu_2|\xi|^2
\end{eqnarray*}
is the uniform parabolicity condition for $\tilde{\mathcal{L}}$ and  the ratio $\frac{r^2\nu_1}{r^2\nu_2}$ is the same as $\frac{\nu_1}{\nu_2}$ and hence we can apply Lemma  \ref{A} for $\tilde{\mathcal{L}}$ and $v$.
Having this in mind, we continue with \eqref{local} as below.

Since 
$$
(t+r^2s,y)\in D_{r}  \quad\Rightarrow \quad (s,y)\in (-1/4,0]\times B_{\frac32 r}(0)
$$
and $(-1/4,0]\times B_{\frac32 r}(0)\subset Q^{\mathcal{D}}_{\frac34}(0,0)$, the last quantity in 
\eqref{local} is bounded by
\begin{equation}\label{change}
N r^{-d+2+2\mu}\int_{Q^{\mathcal{D}}_{\frac34}(0,0)}|y|^{-2\mu-2} |v(s,y)|^2dyds.
\end{equation}
Then we apply Lemma  \ref{A} with $\epsilon_1=\frac34, \epsilon_2=\frac78$ and see
\begin{eqnarray}
\int_{Q^{\mathcal{D}}_{\frac34}(0,0)}|y|^{-2\mu-2} |v(s,y)|^2dyds
&\le& N\int_{Q^{\mathcal{D}}_{\frac78}(0,0)} |v(s,y)|^2dyds\nonumber\\
&\le& N \sup_{Q^{\mathcal{D}}_{\frac78}(0,0)} |v|^2\nonumber\\
&\le& N \sup_{Q^{\mathcal{D}}_{\frac78}(0,0)} |u|^2,\label{sup}
\end{eqnarray}
where the last quantity in \eqref{sup} follows the observation $t+r^2s\in(-\left(\frac78\right)^2,0]$  for any $s\in(-\left(\frac78\right)^2,0]$. 

All the constants $N$ in this Step 2 depend only on $\mathcal{M}$, $\mu$, and $d$. 
Hence, \eqref{local}, \eqref{change} and \eqref{sup}  yield \eqref{spde}, and the claim in Step 1 is proved.
\end{proof}

\begin{remark}
For instance, let $d=3$ and  for any fixed $\kappa\in (0,2\pi)$ take
\begin{eqnarray*}
&&\mathcal{D}=\mathcal{D}_{\kappa}
=\Big\{(r\sin\theta\cos\phi,\ r\sin\theta\sin\phi,\ r\cos\theta)\in\mathbb{R}^3 \mid\nonumber\\
&& \hspace{3.0cm}
  r\in(0,\ \infty),\  0\le\theta<\frac{\kappa}{2}, \ 0<\phi\le 2\pi\Big\}.\label{wedge in 3d}
\end{eqnarray*}
Then the first eigenvalue $\Lambda$ of Laplace-Beltrami operator with the Dirichlet condition on domain $\mathcal{D_{\kappa}}\cap S^{2}$  satisfies
\begin{equation}\label{Lambda estimate}
\frac{1}{2|\log(\cos(\kappa/4))|}\le\Lambda\le  \frac{4j_0^2}{\kappa^2}
\end{equation}
where $j_0\approx 2.4048$ is the first zero of the Bessel function $J_0$ (see \cite{Betz 1983}). Hence, using \eqref{Lambda estimate} and Theorem  \ref{CUB theorem} we can obtain rough  lower bounds of $\lambda^{\pm}_c$.

\end{remark}

\mysection{Evaluation of $\lambda^{\pm}_c$ when $d=2$}\label{sec:Concrete}

Finding the exact values of $\lambda^{\pm}_c$ are very difficult in general.  In Section \ref{sec:critical upper bounds} we presented a decent estimation  of them from below. In this section we  attempt to evaluate    $\lambda^{\pm}_c$ when $d=2$ and the diffusion coefficients $a_{ij}$, $i,j=1,2$, in our operator $\cL$ are constants.

 As $a_{12}=a_{21}$, we can set
\begin{equation*}
A:=(a_{ij})_{2\times 2}:=\begin{pmatrix} a & b\\ b & c \end{pmatrix}.
\end{equation*}
By \eqref{uniform parabolicity}  matrix $A$ is  positive-definite and the eigenvalues are greater than equal to $\nu$ and in particular there is a symmetric matrix $B$ such that $A=B^2$.

For any fixed  $\kappa\in (0,2\pi)$ and $\alpha\in [0,2\pi)$  we denote
$$
\mathcal{D}_{\kappa,\alpha}:=\left\{x=(r\cos\theta,\ r\sin\theta)\in\mathbb{R}^2 \,|\, r\in(0,\ \infty),\ -\frac{\kappa}{2}+\alpha<\theta<\frac{\kappa}{2}+\alpha\right\},
$$
calling $\kappa$ the central  angle of the domain $\cD_{\kappa,\alpha}$. 

We consider the operator 
$$
\cL=\frac{\partial}{\partial t}-\big(aD_{x_1x_1}+b(D_{x_1x_2}+D_{x_2x_1})+cD_{x_2x_2}\big)
$$
with the conic (angular) domain $\cD_{\kappa,\alpha}$.

Below $\arctan$ is a map from $\bR \to (-\pi/2, \pi/2)$.

\begin{prop}
For $\cL$ and $\cD_{\kappa,\alpha}$ defined above, we have
\begin{align*}
\lambda^{\pm}_c(\cL,\cD_{\kappa,\alpha})=\frac{\,\pi\,}{\widetilde{\kappa}},
\end{align*}
where
\begin{equation}\label{new kappa}
\widetilde{\kappa}=\pi-\arctan\Big(\,\frac{\bar{c}\,\cot(\kappa/2)+\bar{b}}{\sqrt{\det(A)}}\,\Big)-\arctan\Big(\,\frac{\bar{c}\,\cot(\kappa/2)-\bar{b}}{\sqrt{\det(A)}}\,\Big)
\end{equation}
with constants $\bar{a}, \bar{b}$ from the relation
\begin{equation}\label{bar a bar b}
\begin{pmatrix} \bar{a} & \bar{b}\\
\bar{b}& \bar{c} \end{pmatrix}
= \begin{pmatrix} \cos \alpha & \sin \alpha\\
-\sin \alpha & \cos \alpha \end{pmatrix} 
\begin{pmatrix} a & b\\
b& c  \end{pmatrix} 
\begin{pmatrix} \cos \alpha & - \sin \alpha\\
\sin \alpha & \cos \alpha \end{pmatrix}.
\end{equation}
In particular, 

(i) if $\kappa=\pi$, then $\tilde{\kappa}=\pi$;

(ii) if $\alpha=0$ and $b=0$, then $\tilde{\kappa}$ is determined by the relation
\begin{equation}\label{new kappa simpler case of operator}
\tan\Big(\frac{\widetilde{\kappa}}{\,2\,}\Big)=\sqrt{\frac{a}{c}}\tan\Big(\frac{\kappa}{\,2\,}\Big)
\end{equation}
for $\kappa\in(0,2\pi)\setminus\{\pi\}$. 
\end{prop}

\begin{proof}

\textbf{1}.    We first consider the operator 
$$
\mathcal{L}_0:=\frac{\partial}{\partial t}-\Delta_x
$$
with domain $\cD_{\kappa,\alpha}$.  In this case we note $\tilde{\kappa}=\kappa$ and, as  in Remark \ref{improvement on critical lambda}, we again have
 $$
 \lambda^+_c=\lambda^-_c=\sqrt{\Lambda}=\frac{\pi}{\kappa}.
 $$
Indeed,  the eigenvalue/eigenfunction  problem 
$$
- \phi''(\theta)=\lambda \phi(\theta), \quad \theta\in \Big(-\frac{\kappa}{2}+\alpha,\, -\frac{\kappa}{2}+\alpha\Big) \,; \quad \quad \phi\left(-\frac{\kappa}{2}+\alpha\right)=\phi\left(\frac{\kappa}{2}+\alpha\right)=0
$$
 leads us to have $\phi(\theta)=\cos\big(\sqrt{\lambda}(\theta-\alpha)\big)$ and $\cos\big(\sqrt{\lambda}\ \kappa/2\big)=0$. Hence, the first eigenvlaue $\Lambda$ again satisfies $\sqrt{\Lambda}\ \kappa/2=\pi/2$. Thus we have
  $$
\lambda^{\pm}_c(\cL_0,\cD_{\kappa,\alpha})=\sqrt{\Lambda}=\frac{\pi}{\tilde{\kappa}}.
$$

\textbf{2}. General case. Having \eqref{CUB4} and the accompanied explanation in mind, we take a symmetric matrix $B$ satisfying $A=B^2$.
The change of variables $x=By$ transforms the operator $aD_{x_1x_1}+bD_{x_1x_2}+bD_{x_2x_1}+cD_{x_2x_2}$ into $\Delta_y=D_{y_1y_1}+D_{y_2y_2}$ in $y$-coordinates, that is, putting $v(t,y)=u(t,By)$,
we obtain
$$
\big(aD_{11}u+bD_{12}u+b D_{21}u+cD_{22}u)(t,By)=\Delta_y v(t,y),\quad (t,y)\in \bR\times \widetilde{\cD}\, ,
$$
where $\widetilde{\cD}$ is the image of $\cD_{\kappa,\alpha}$ under a linear transformation defined by
$$
\widetilde{\cD}:=B^{\,-1}\cD_{\kappa,\alpha}:=\left\{B^{\,-1}x\,:\,x\in\cD_{\kappa,\alpha}\right\}\, .
$$

We note that $\widetilde{\cD}$ is also a conic (angular) domain with a certain central angle $\widetilde{\kappa}$. In fact, we can use \eqref{CUB4} and Step 1 to  have
$$
\lambda^{\pm}_c(\cL,\cD_{\kappa,\alpha})=\lambda^{\pm}_c(\cL_0,\widetilde{\cD})=\frac{\pi}{\,\widetilde{\kappa}\,}.
$$

Let us  verify the formula for $\,\widetilde{\kappa}$.  We first note
\begin{align*}
\frac{\widetilde{\kappa}}{2\pi}=\frac{|\widetilde{\cD}\cap B_1(0)|_{\ell}}{|B_1(0)|_{\ell}}\quad\text{and hence}\quad \widetilde{\kappa}=2\cdot|\widetilde{\cD}\cap B_1(0)|_{\ell},
\end{align*}
where $|E|_{\ell}$ denotes the Lebesgue measure of $E\subset\bR^2$. 
By the  relation $y=B^{-1}x$, we then have

\begin{align*}
|\widetilde{\cD}\cap B_1(0)|_{\ell}\,&=\int_{\left\{y\,\in \widetilde{\cD}\,:\,|y|\leq 1\right\}}\,dy\\
&=\frac{1}{|\det(B)|}\int_{\left\{x\,\in\cD\,:\,|B^{-1}x|\leq 1\right\}}dx\\
&=\frac{1}{\sqrt{\det(A)}}\int_{-\kappa/2+\alpha}^{\kappa/2+\alpha}\int_0^{|B^{-1}v_{\theta}|^{-1}}r\,dr\,d\theta\\
&=\frac{1}{2\sqrt{\det(A)}}\int_{-\kappa/2+\alpha}^{\kappa/2+\alpha}\frac{1}{\,|B^{\,-1}v_{\theta}|^2\,}\,d\theta\\
&=\frac{1}{2\sqrt{\det(A)}}\int_{-\kappa/2+\alpha}^{\kappa/2+\alpha}\frac{1}{\,v_{\theta}^TA^{\,-1}v_{\theta}\,}\,d\theta,
\end{align*}
where $v_\theta:=\begin{pmatrix} \cos \theta \\\sin \theta \end{pmatrix}$.
Now, a direct calculation based on translation, symmetry, and change of variable gives
\begin{align*}
&\quad\;|\widetilde{\cD}\cap B_1(0)|_{\ell} \\
&=\frac{1}{2\sqrt{\det(A)}}\left(\int_{0}^{\kappa/2}\frac{1}{\,v_{\theta}^T\,\overline{A}^{\,-1}\,v_{\theta}\,}\,d\theta+\int_{-\kappa/2}^{0}\frac{1}{\,v_{\theta}^T\,\overline{A}^{\,-1}\,v_{\theta}\,}\,d\theta\right)\\
&=\frac{\sqrt{\det(A)}}{2}\int_{0}^{\kappa/2}\left(\frac{1}{\,\bar{c}\cot^2\theta-2\bar{b}\cot\theta +\bar{a}}+\frac{1}{\,\bar{c}\cot^2\theta+2\bar{b}\cot\theta +\bar{a}}\right)\cdot\frac{1}{\sin^2\theta}\,d\theta\\
&=\frac{\sqrt{\det(A)}}{2}\int_{\cot(\kappa/2)}^{\infty}\frac{1}{\,\bar{c}\,t^2-2\bar{b}\,t +\bar{a}}+\frac{1}{\,\bar{c}\,t^2+2\bar{b}\,t +\bar{a}}dt\\
&=\frac{1}{2}\left(\pi-\arctan\Big(\,\frac{\bar{c}\,\cot(\kappa/2)-\bar{b}}{\sqrt{\det(A)}}\,\Big)-\arctan\Big(\,\frac{\bar{c}\,\cot(\kappa/2)+\bar{b}}{\sqrt{\det(A)}}\,\Big)\right)\, ,
\end{align*}
where
$$
\overline{A}=
\begin{pmatrix}
\bar{a} & \bar{b} \\
\bar{b} & \bar{c}
\end{pmatrix}
$$
with $\bar{a}$, $\bar{b}$, and $\bar{c}$ defined in \eqref{bar a bar b}.
Hence, we obtain \eqref{new kappa} for $\tilde{\kappa}$ and the proof is done.
\end{proof}

\begin{remark}
Let us consider the simple but essential case of  $b=0$ and $\alpha=0$, i.e., $\cL$ with $A=\begin{pmatrix} a & 0 \\0 & c\end{pmatrix}$ and domain $\cD_{\kappa}$. Then, from \eqref{new kappa simpler case of operator}, we observe that the  ratio $r:=\frac{a}{c}$ of the diffusion constants, rather than the exact values of $a$ and $c$, along with $\kappa$ decides $\tilde{\kappa}$ and hence the values $\lambda_c^{\pm}$.
We also note that for $\kappa\in(0,\pi)$
$$
\tilde{\kappa}\to\pi^-\quad\text{as}\quad r\to \infty\,;\quad
\tilde{\kappa}\to 0^+\quad\text{as}\quad r\to 0^+
$$
and  for $\kappa\in(\pi,2\pi)$
$$
\tilde{\kappa}\to\pi^+\quad\text{as}\quad r\to \infty\,;\quad
\tilde{\kappa}\to 2\pi^-\quad\text{as}\quad r\to 0^+.
$$
In particular, if $\kappa\in(0,\pi)$, or domain $\cD_{\kappa}$ is convex, and the diffusion constant to $x_2$ direction is relatively much  lager than the  the diffusion constant to $x_1$ direction, then $\lambda_c^{\pm}$ are much bigger  than 1 and hence Green's function estimate \eqref{key estimate} gives better decay near the vertex since $R_{t,x}\le 1$.
\end{remark}



\end{document}